\documentclass[leqno,11pt]{amsart} 
\usepackage{amssymb}
\theoremstyle{plain} 
\newtheorem{theorem}{\indent\sc Theorem}[section]

\newtheorem{corollary}[theorem]{\indent\sc Corollary}
\newtheorem{proposition}[theorem]{\indent\sc Proposition}

\newtheorem{problem}[theorem]{\indent\sc Problem}
\theoremstyle{definition} 

\newtheorem{remark}[theorem]{\indent\sc Remark}

\def\C{{\mathbf{C}}}
\def\R{{\mathbf{R}}}
\def\H{{\mathbf{H}}}
\def\Z{{\mathbf{Z}}}
\def\Pi{{\mathbf{P}}}
\def\Lc{{\mathcal{L}}}
\def\RC{{\overline{\mathbf{C}}}} 

\begin{document}

\title[Value distribution of Gauss maps]{Value distribution for the Gauss maps of \\ various classes of surfaces} 

\author[Y. Kawakami]{YU KAWAKAMI}

\dedicatory{Dedicated to Professor Ryoichi Kobayashi on the occation of his sixtieth birthday}

\renewcommand{\thefootnote}{\fnsymbol{footnote}}
\footnote[0]{2010\textit{ Mathematics Subject Classification}.
Primary 53A10; Secondary 30D35, 53C42.}
%
%
\thanks{ 
Partly supported by the Grant-in-Aid for Scientific Research (C), 
No. 15K04840, Japan Society for the Promotion of Science.}
\address{
Faculty of Mathematics and Physics, \endgraf
Institute of Science and Engineering, \endgraf 
Kanazawa University, \endgraf
Kanazawa, 920-1192, Japan
}
\email{y-kwkami@se.kanazawa-u.ac.jp}


\maketitle

\begin{abstract}
We present in this article a survey of recent results in value distribution theory for the Gauss maps 
of several classes of immersed surfaces in space forms, for example, minimal surfaces 
in Euclidean $n$-space ($n$=3 or 4), improper affine spheres in the affine 3-space and flat surfaces in 
hyperbolic 3-space. In particular, we elucidate the geometric background of their results. 
\end{abstract}

\section{Introduction} 
The geometric nature of value distribution theory of complex analytic maps is well-known. 
One of the most notable results is the geometric interpretation of the precise maximum `2' for the number 
of exceptional values of a nonconstant meromorphic function on the complex plane $\C$. Here we call a value that 
a function or map never attains an {\it exceptional value} of the function or map. In fact, Ahlfors \cite{Ah1935} and 
Chern \cite{Ch1960} proved that the least upper bound for the number of exceptional values of a nonconstant holomorphic map 
from $\C$ to a closed Riemann surface $\overline{\Sigma}_{\gamma}$ of genus $\gamma$ coincides with the Euler characteristic 
of $\overline{\Sigma}_{\gamma}$ by using Nevanlinna theory (see also \cite{Ko2003, Ne1970, NO1990, NW2013, Ru2001}). 
In particular, for a nonconstant meromorphic function on $\C$, the geometric meaning 
of the maximal number `2' of exceptional values is the Euler characteristic of the Riemann sphere $\RC :=\C\cup \{\infty\}$. 
We remark that if the closed Riemann surface is of $\gamma \geq 2$, then such a map does not exist because the Euler characteristic 
is negative. 

There exist several classes of immersed surfaces in 3-dimensional space forms whose Gauss maps have value-distribution-theoretic 
property. For instance, Fujimoto (\cite[Theorem I]{Fu1988}, \cite{Fu1997}) proved that the Gauss map of a nonflat 
complete minimal surface in Euclidean 3-space ${\R}^{3}$ can omit at most 4 values. Moreover, Fujimoto \cite{Fu1993} 
obtained a unicity theorem for the Gauss maps of 
nonflat complete minimal surfaces in ${\R}^{3}$, which is analogous to the Nevanlinna unicity theorem (\cite{Ne1926}) for meromorphic 
functions on $\C$. 
On the other hand, the author and Nakajo \cite{KN2012} showed that the maximal number of exceptional 
values of the Lagrangian Gauss map of a weakly complete improper affine front in the affine 3-space is 3, unless it is an elliptic 
paraboloid. Moreover, the author \cite{Ka2014} gave a similar result for flat fronts in hyperbolic 3-space ${\H}^{3}$. 

The purpose of this review paper is to give geometric interpretation of value-distribution-theoretic property for their Gauss maps. 
The paper is organized as follows: In Section \ref{M-results}, we first give a curvature bound for the conformal metric 
$ds^{2}=(1+|g|^{2})^{m}|\omega|^{2}$ on an open Riemann surface $\Sigma$, where $m$ is a positive integer, $\omega$ is a holomorphic 
1-form and $g$ is a meromorphic function $g$ on $\Sigma$ (Theorems \ref{thm2-1} and \ref{thm2-2}). 
As a corollary of it, we prove that the 
precise maximum for the number of exceptional values of $g$ on $\Sigma$ with the complete conformal metric $ds^{2}$ is `$m+2$' 
(Corollary \ref{thm2-3}). We note that the geometric interpretation of the `2' in `$m+2$' is the Euler characteristic of the 
Riemann sphere (Remark \ref{rmk2-1-1}). We also give a unicity theorem (Theorem \ref{thm2-5}) for $g$ on $\Sigma$ with the 
complete metric $ds^{2}$. In Section \ref{appli}, as application of our main results, we show some value-distribution-theoretic 
properties for the Gauss maps of the following classes of surfaces: minimal surfaces in ${\R}^{3}$ (Section \ref{appli-mini}), 
improper affine fronts in ${\R}^{3}$ (Section \ref{appli-improper}) and flat fronts in ${\H}^{3}$ (Section \ref{appli-flat}). 
Section \ref{sec-further} is denoted value-distribution-theoretic property for the Gauss map of a complete minimal surface 
in ${\R}^{4}$ (Section \ref{further-4}) and a complete minimal surface of finite total curvature in ${\R}^{3}$ 
(Section \ref{further-algebraic}). 

\section{Main results}\label{M-results}
We first give the following curvature bound for the conformal metric 
$$
ds^{2}=(1+|g|^{2})^{m}|\omega|^{2}
$$
on an open Riemann surface $\Sigma$. 

\begin{theorem}[\cite{Ka2013}]\label{thm2-1}
Let $\Sigma$ be an open Riemann surface with the conformal metric 
\begin{equation}\label{eq2-1}
ds^{2}=(1+|g|^{2})^{m}|\omega|^{2},  
\end{equation}
where $\omega$ is a holomorphic $1$-from, $g$ is a meromorphic function on $\Sigma$, and $m$ is a positive 
integer. Assume that $g$ omits $q\geq m+3$ distinct values. Then there exists a positive constant $C$, depending 
on $m$ and the set of exceptional values, but not the surface, such that for all $p\in \Sigma$, we have 
\begin{equation}\label{eq2-2}
|K_{ds^{2}}(p)|^{1/2}\leq \dfrac{C}{d(p)}, 
\end{equation} 
where $K_{ds^{2}}(p)$ is the Gaussian curvature of $ds^{2}$ at $p$ and $d(p)$ is the geodesic distance from $p$ to 
the boundary of $\Sigma$, that is, the infimum of the lengths of the divergent curves in $\Sigma$ emanating from $p$. 
\end{theorem}

More generally, when all of the multiple values of the meromorphic function $g$ in the metric (\ref{eq2-1}) are totally 
ramified, the following theorem holds. 

\begin{theorem}[\cite{Ka2015}]\label{thm2-2}
Let $\Sigma$ be an open Riemann surface with the conformal metric given by (\ref{eq2-1}). Let $q$ be a positve integer, 
${\alpha}_{1}, \ldots, {\alpha}_{q}\in \RC$ be distinct and 
${\nu}_{1}, \ldots, {\nu}_{q}\in {\Z}_{+}\cup \{\infty\}$. Assume that 
\begin{equation}\label{eq2-3}
\gamma := \displaystyle \sum_{j=1}^{q} \biggl{(}1-\dfrac{1}{{\nu}_{j}} \biggr{)}> m+2. 
\end{equation}
If $g$ satisfies the property that all ${\alpha}_{j}$-points of $g$ have multiplicity at least ${\nu}_{j}$, 
then  there exists a positive constant $C$, depending on $m$, $\gamma$ and ${\alpha}_{1}, \ldots, {\alpha}_{q}$ but 
not the surface, such that for all $p\in \Sigma$ inequality (\ref{eq2-2}) holds. 
\end{theorem}

This is a generalization of Theorem \ref{thm2-1}. Indeed, we can show it by setting ${\nu}_{1}=\cdots ={\nu}_{q=m+3}= \infty$. 
As an application of this theorem, we obtain an analogue of a special case of the Ahlfors islands theorem 
(See \cite{Be2000} for details of this theorem) for 
the meromorphic function $g$ on $\Sigma$ with the complete conformal metric $ds^{2}$. For more details, see \cite{Ka2015}. 

As a corollary of these theorems, we give the least upper bound for the number of exceptional values of the meromorphic 
function $g$ on $\Sigma$ with the complete conformal metric given by (\ref{eq2-1}). 

\begin{corollary}[\cite{Ka2013}]\label{thm2-3}
Let $\Sigma$ be an open Riemann surface with the conformal metric given by (\ref{eq2-1}). If the metric $ds^{2}$ 
is complete and the meromorphic function $g$ is nonconstant, then $g$ can omit at most $m+2$ distinct values. 
\end{corollary}
\begin{proof}
By way of contradiction, suppose that $g$ omits $m+3$ distinct values. By Theorem \ref{thm2-1}, (\ref{eq2-1}) holds. 
If $ds^{2}$ is complete, then we may set $d(p)= +\infty$ for all $p\in \Sigma$. Thus $K_{ds^{2}}\equiv 0$ on $\Sigma$. 
On the other hand, the Gaussian curvature with respect to the metric $ds^{2}$ is given by 
$$
K_{ds^{2}}=-\dfrac{2m|g_{z}'|^{2}}{(1+|g|^{2})^{m+2}|\hat{\omega}_{z}|^{2}},  
$$
where $\omega =\hat{\omega}_{z}dz$, $g'_{z}=dg/dz$. Thus $K_{ds^{2}}\equiv 0$ if and only if $g$ is constant. 
This contradicts the assumption that $g$ is nonconstant. 
\end{proof}

We give examples which ensure that Corollary \ref{thm2-3} is optimal. 
\begin{proposition}\label{thm2-4}
Let $\Sigma$ be either the complex plane punctured at $q-1$ distinct points ${\alpha}_{1}, \ldots, {\alpha}_{q-1}$ or 
the universal cover of that punctured plane. We set 
$$
\omega=\dfrac{dz}{{\prod}_{i=1}^{q-1}(z-{\alpha}_{i})}, \quad g=z. 
$$
Then $g$ omits $q$ distinct values and the metric $ds^{2}=(1+|g|^{2})^{m}|\omega|^{2}$ is complete if and only if 
$q\leq m+2$. In particular, there exist examples whose metric $ds^{2}$ is complete and $g$ omits $m+2$ distinct values. 
\end{proposition}
\begin{proof}
We can easily show that $g$ omits the $q$ distinct values ${\alpha}_{1}, \ldots, {\alpha}_{q-1}$ and $\infty$ on 
$\Sigma$. A divergent curve $\Gamma$ in $\Sigma$ must tend to one of the points 
${\alpha}_{1}, \ldots, {\alpha}_{q-1}$ or $\infty$.   
Thus we have 
$$
\int_{\Gamma}ds=\int_{\Gamma}(1+|g|^{2})^{m/2}|\omega|=\int_{\Gamma}\dfrac{(1+|z|^{2})^{m/2}}{\prod_{i=1}^{q-1}|z-{\alpha}_{i}|}|dz|= +\infty, 
$$
when $q\leq m+2$. 
\end{proof}

\begin{remark}\label{rmk2-1-1}
The geometric interpretation of the `$2$' in `$m+2$' is the Euler characteristic of the Riemann sphere. 
Indeed, if $m=0$ then the metric $ds^{2}=(1+|g|^{2})^{0}|\omega|^{2}=|\omega|^{2}$ is flat and complete on $\Sigma$. 
We thus may assume that $g$ is a meromorphic function on ${\C}$ because $g$ is replaced by $g\circ \pi$, 
where $\pi\colon \C\to \Sigma$ is a holomorphic universal covering map. 
On the other hand, Ahlfors \cite{Ah1935} and Chern \cite{Ch1960} showed that the best possible upper bound `$2$' 
of the number of exceptional values of nonconstant meromorphic functions on $\C$ coincides with the Euler 
characteristic of the Riemann sphere. Hence we get the conclusion. Remark that Ros \cite{Ro2002} gave a different 
approach of this fact by using `Bloch-Zalcman principle'. 
\end{remark}

We next provide another type of value-distribution-theoretic property of the meromorphic function $g$ on an open Riemann surface $\Sigma$ 
with the conformal metric given by (\ref{eq2-1}). Nevanlinna \cite{Ne1926} showed that two nonconstant meromorphic functions on $\C$ having the same images for 5 distinct values must identically equal to each other. We obtain the following analogue to this unicity theorem. 
\begin{theorem}[\cite{Ka2015}]\label{thm2-5}
Let $\Sigma$ be an open Riemann surface with the conformal metric 
\begin{equation}\label{eq2-4}
ds^{2}=(1+|g|^{2})^{m}|\omega|^{2}, 
\end{equation}
and $\widehat{\Sigma}$ another open Riemann surface with the conformal metric 
\begin{equation}\label{eq2-5}
d{\hat{s}}^{2}=(1+|\hat{g}|^{2})^{m}|\hat{\omega}|^{2}, 
\end{equation}
where $\omega$ and $\hat{\omega}$ are holomorphic $1$-forms, $g$ and $\hat{g}$ are nonconstant meromorphic functions 
on $\Sigma$ and $\widehat{\Sigma}$ respectively, and $m$ is a positive integer. 
We assume that there exists a conformal diffeomorphism $\Psi\colon \Sigma \to \widehat{\Sigma}$. Suppose that there exist $q$ distinct points 
${\alpha}_{1}, \ldots, {\alpha}_{q}\in \RC$ such that $g^{-1}({\alpha}_{j})=(\hat{g}\circ \Psi)^{-1}({\alpha}_{j})$ $(1\leq j\leq q)$. 
If $q \geq m+5 \,(=(m+4)+1)$ and either $ds^{2}$ or $d{\hat{s}}^{2}$ is complete, then $g\equiv \hat{g}\circ \Psi$. 
\end{theorem}

We remark that Theorem \ref{thm2-5} coincides with the Nevanlinna unicity theorem when $m=0$. 
The maps $g$ and $\hat{g}\circ \Psi$ are said to share the value $\alpha$ (ignoring multiplicity) when $g^{-1}(\alpha)= (\hat{g}\circ \Psi)^{-1}(\alpha)$. 
Theorem \ref{thm2-5} is optimal for an arbitrary even number $m\, (\geq 2)$ because there exist the following examples. 

\begin{proposition}[\cite{Fu1993, Ka2015}]\label{thm2-6}
For an arbitrary even number $m\, (\geq 2)$, we take $m/2$ distinct points ${\alpha}_{1}, \ldots, {\alpha}_{m/2}$ in $\C\backslash \{0, \pm 1 \}$. 
Let $\Sigma$ be either the complex plane punctured at $m+1$ distinct points $0$, ${\alpha}_{1}, \ldots, {\alpha}_{m/2}$, 
$1/{\alpha}_{1}, \ldots, 1/{\alpha}_{m/2}$ or the universal covering of that punctured plane. We set 
$$
\omega = \dfrac{dz}{z\prod_{i=1}^{m/2} (z-{\alpha}_{i})({\alpha}_{i}z-1)}, \quad g(z)=z\,,
$$
and 
$$
\hat{\omega}\, (=\omega) = \dfrac{dz}{z\prod_{i=1}^{m/2} (z-{\alpha}_{i})({\alpha}_{i}z-1)}, \quad \hat{g}(z)=\dfrac{1}{z}. 
$$
We can easily show that the identity map $\Psi\colon \Sigma \to \Sigma$ is a conformal diffeomorphism and 
the metrics $ds^{2}=(1+|g|^{2})^{m}|\omega|^{2}$ is complete. Then the maps $g$ and $\hat{g}$ share the $m+4$ distinct values 
$$ 0,\, \infty,\, 1,\, -1,\, {\alpha}_{1},\,\ldots, {\alpha}_{m/2},\, 1/{\alpha}_{1},\, \ldots,\, 1/{\alpha}_{m/2}$$ 
and $g\not\equiv \hat{g}\circ \Psi$. These show that the number `$m+5$' in Theorem \ref{thm2-5} cannot be replaced by `$m+4$'. 
\end{proposition}

\section{Applications}\label{appli}
In this section, as applications of the main results, we give some value-distribution-theoretic properties for the Gauss maps of 
several classes of surfaces. 

\subsection{Gauss map of a complete minimal surface in ${\R}^{3}$}\label{appli-mini}
We first recall some basic facts of minimal surfaces in Euclidean 3-space ${\R}^{3}$. Details can be found, for example, 
in \cite{Fu1993-2}, \cite{La1982} and \cite{Os1986}. Let $X=(x^{1}, x^{2}, x^{3})\colon \Sigma \to {\R}^{3}$ be an oriented minimal surface in ${\R}^{3}$. 
By associating a local complex coordinate $z=u+\sqrt{-1}v$ with each positive isothermal coordinate system $(u, v)$, $\Sigma$ is considered as 
a Riemann surface whose conformal metric is the induced metric $ds^{2}$ from ${\R}^{3}$. Then 
\begin{equation}\label{equ-411-Laplacian}
{\triangle}_{ds^{2}}X=0
\end{equation}
holds, that is, each coordinate function $x^{i}$ is harmonic. With respect to the local complex coordinate $z=u+\sqrt{-1}v$ of the surface, 
(\ref{equ-411-Laplacian}) is given by 
$$
\bar{\partial}\partial X = 0, 
$$
where $\partial =(\partial /\partial u -\sqrt{-1}\partial /\partial v)/2$, $\bar{\partial} =(\partial /\partial u +\sqrt{-1}\partial /\partial v)/2$. 
Hence each ${\phi}_{i}:=\partial x^{i}dz$ $(i=1, 2, 3)$ is a holomorphic 1-form on $\Sigma$. If we set that 
\begin{equation}\label{eq:W-date3}
\omega ={\phi}_{1}-\sqrt{-1}{\phi}_{2}, \quad g=\dfrac{{\phi}_{3}}{{\phi}_{1}-\sqrt{-1}{\phi}_{2}}, 
\end{equation}
then $\omega$ is a holomorphic 1-form and $g$ is a meromorphic function on $\Sigma$. Moreover the function $g$ coincides with the composition of 
the Gauss map and the stereographic projection from ${\mathbf{S}}^{2}$ onto $\RC$, and the induced metric is given by 
\begin{equation}\label{equ-412-metric}
ds^{2}=(1+|g|^{2})^{2}|\omega|^{2}. 
\end{equation}

Applying Theorem \ref{thm2-1} to the metric $ds^{2}$, we can obtain the Fujimoto curvature bound 
for a minimal surface in ${\R}^{3}$. 

\begin{theorem}\cite[Theorem I and Corollary 3.4]{Fu1988}\label{thm3-1} 
Let $X\colon \Sigma \to {\R}^{3}$ be an oriented minimal surface whose Gauss map $g\colon \Sigma \to \RC$ omits greater than or 
equal to $5$ $(=2+3)$ distinct values. Then there exists a positive constant $C$ depending on the set of exceptional values, but not the surface, 
such that for all $p\in \Sigma$ inequality (\ref{eq2-2}) holds. In particular, the Gauss map of a nonflat complete minimal surface in ${\R}^{3}$ 
can omit at most $4$ $(=2+2)$ values. 
\end{theorem}

We note that this theorem is a generalization of the Bernstein theorem, stating that the only solution to the minimal surface equation over the whole plane is the trivial solution: a linear function (\cite{Be1915, Ca1970}). 

\begin{remark}\label{thm3-2-1}
For the Gauss maps of complete embedded minimal surfaces in ${\R}^{3}$, 
there exists an interesting conjecture called `Four Point Conjecture'. For more details, see \cite{MP2012}. 
\end{remark}

Moreover, by applying Theorem \ref{thm2-5}, we can get the Fujimoto unicity theorem for the Gauss maps of complete minimal surfaces in ${\R}^{3}$. 

\begin{theorem}[{\cite[Theorem I]{Fu1993-2}}]\label{thm3-2}
Let $X\colon \Sigma \to {\R}^{3}$ and $\widehat{X}\colon \widehat{\Sigma}\to {\R}^{3}$ be two nonflat minimal surfaces and 
assume that there exists a conformal diffeomorphism $\Psi\colon \Sigma \to \widehat{\Sigma}$. Let $g\colon \Sigma \to \RC$ and 
$\hat{g}\colon \widehat{\Sigma}\to \RC$ be the Gauss maps of $X(\Sigma)$ and $\widehat{X}(\widehat{\Sigma})$, respectively. 
If $g\not\equiv \hat{g}\circ \Psi$ and either $X(\Sigma)$ or $\widehat{X}(\widehat{\Sigma})$ is complete, 
then $g$ and $\hat{g}\circ \Psi$ share at most $6\,(=2+4)$ distinct values. 
\end{theorem}

\subsection{Lagrangian Gauss map of a weakly complete improper affine front in ${\R}^{3}$}\label{appli-improper}
Improper affine spheres in the affine $3$-space ${\R}^{3}$ also have similar properties to minimal surfaces in 
Euclidean $3$-space (for example, see \cite{Ca1970}). 
Recently, Mart\'inez \cite{Ma2005} discovered the correspondence between improper affine spheres and 
smooth special Lagrangian immersions in 
the complex $2$-space ${\C}^{2}$ and introduced the notion of {\it improper affine fronts}, that is, 
a class of (locally strongly convex) improper affine spheres with some admissible singularities in ${\R}^{3}$. 
We note that this class is called 
`improper affine maps' in \cite{Ma2005}, but we call this class `improper affine fronts' because all of improper affine maps are 
wave fronts in ${\R}^{3}$ (\cite{Na2009}, \cite{UY2011}). 
The differential geometry of wave fronts is discussed in \cite{SUY2009}. 
Moreover, Mart\'inez gave the following holomorphic representation for this class. 

\begin{theorem}[{\cite[Theorem 3]{Ma2005}}]
Let $\Sigma$ be a Riemann surface and $(F, G)$ a pair of holomorphic functions on $\Sigma$ such that $\text{Re}(FdG)$ is exact and 
$|dF|^{2}+|dG|^{2}$ is positive definite. Then the induced map $\psi\colon \Sigma \to {\R}^{3}=\C\times {\R}$ given by 
$$
\psi :=\biggl{(}G+\overline{F}, \dfrac{|G|^{2}-|F|^{2}}{2}+\text{Re}\biggl{(} GF- 2\int FdG \biggr{)} \biggr{)}
$$
is an improper affine front. Conversely, any improper affine front is given in this way. 
Moreover we set $x:= G+\overline{F}$ and $n:= \overline{F}-G$. Then $L_{\psi}:=x+\sqrt{-1}n\colon \Sigma \to {\C}^{2}$ is a special Lagrangian 
immersion whose induced metric $d{\tau}^{2}$ from ${\C}^{2}$ is given by 
$$
d{\tau}^{2}=2(|dF|^{2}+|dG|^{2}). 
$$
In addition, the affine metric $h$ of $\psi$ is expressed as $h:=|dG|^{2}-|dF|^{2}$ and the singular points of $\psi$ correspond to the 
points where $|dF|=|dG|$. 
\end{theorem}

We remark that Nakajo \cite{Na2009} constructed a representation formula for indefinite improper affine spheres with 
some admissible singularities. The nontrivial part of the Gauss map of $L_{\psi}\colon \Sigma \to {\C}^{2}\simeq {\R}^{4}$ (see \cite{CM1987}) is the meromorphic function 
$\nu\colon \Sigma \to \RC$ given by 
$$
\nu := \dfrac{dF}{dG}, 
$$
which is called the {\it Lagrangian Gauss map} of $\psi$. An improper affine front is said to be {\it weakly complete} if the induced metric 
$d{\tau}^{2}$ is complete. We note that 
$$
d{\tau}^{2}=2(|dF|^{2}+|dG|^{2})=2(1+|\nu|^{2})|dG|^{2}. 
$$

Applying Theorem \ref{thm2-1} to the metric $d{\tau}^{2}$, we can get the following theorem. 

\begin{theorem}[{\cite[Theorem 4.6]{Ka2013}}]\label{thm3-3}
Let $\psi\colon \Sigma \to {\R}^{3}$ be an improper affine front whose Lagrangian Gauss map $\nu\colon \Sigma \to \RC$ 
omits greater than or equal to $4$ $(=2+2)$ distinct values. 
Then there exists a positive constant $C$ depending on the set of exceptional values, but not $\Sigma$, such that for all $p\in \Sigma$ we have 
$$
|K_{d{\tau}^{2}}(p)|^{1/2}\leq \dfrac{C}{d(p)}, 
$$
where $K_{d{\tau}^{2}}(p)$ is the Gaussian curvature of the metric $d{\tau}^{2}$ at $p$ and $d(p)$ is the geodesic distance from $p$ to  
the boundary of $\Sigma$. In particular, if the Lagrangian Gauss map of a weakly complete improper affine front in ${\R}^{3}$ is nonconstant, 
then it can omit at most $3$ $(=1+2)$ values. 
\end{theorem}

Since the singular points of $\psi$ correspond to the points where $|\nu|=1$, we can obtain a simple proof of the parametric affine Bernstein 
theorem (\cite{Ca1958}, \cite{Jo1954}) for improper affine spheres from the viewpoint of value-distribution-theoretic properties of the Lagrangian Gauss map. 
\begin{corollary}[{\cite{Ca1958, Jo1954}}]\label{them3-4}
Any affine complete improper affine sphere in ${\R}^{3}$ must be an elliptic paraboloid. 
\end{corollary}
\begin{proof}
Since an improper affine sphere has no singularities, the complement of the image of its Lagrangian Gauss map $\nu$ contains at least 
the circle $\{|\nu| =1\}\subset \RC$. Thus, by exchanging roles of $dF$ and $dG$ if necessarily, $|\nu|< 1$ holds, that is, 
$|dF|<|dG|$. 
On the other hand, we have 
$$
h= |dG|^{2}-|dF|^{2}< 2(|dF|^{2}+|dG|^{2}) = d{\tau}^{2}. 
$$
Thus if an improper affine sphere is affine complete, then it is also weakly complete. From Theorem \ref{thm3-3} and \cite[Proposition 3.1]{KN2012}, 
 it is an elliptic paraboloid. 
\end{proof}

By applying Theorem \ref{thm2-5}, we give the following unicity theorem for the Lagrangian Gauss maps of weakly complete improper affine fronts in 
${\R}^{3}$. 

\begin{theorem}[{\cite[Theorem 4.24]{Ka2015}}]\label{thm3-5}
Let $\psi\colon \Sigma \to {\R}^{3}$ and $\widehat{\psi}\colon \widehat{\Sigma}\to {\R}^{3}$ be two improper affine fronts and 
assume that there exists a conformal diffeomorphism $\Psi\colon \Sigma \to \widehat{\Sigma}$. Let $\nu\colon \Sigma \to \RC$ and 
$\hat{\nu}\colon \widehat{\Sigma}\to \RC$ be the Lagrangian Gauss maps of $\psi(\Sigma)$ and $\widehat{\psi}(\widehat{\Sigma})$ respectively. 
Suppose that there exist $q$ distinct points ${\alpha}_{1}, \ldots, {\alpha}_{q}\in \RC$ 
such that ${\nu}^{-1}({\alpha}_{j})=(\hat{\nu}\circ \Psi)^{-1}({\alpha}_{j})$ $(1\leq j\leq q)$. 
If $q \geq 6 \,(=(1+4)+1)$ and either $\psi (\Sigma)$ or $\widehat{\psi} (\widehat{\Sigma})$ is weakly complete, then either 
$\nu\equiv \hat{\nu}\circ \Psi$ or $\nu$ and $\hat{\nu}$ are both constant, that is, $\psi (\Sigma)$ and  $\widehat{\psi} (\widehat{\Sigma})$ 
are both elliptic paraboloids. 
\end{theorem}

\subsection{Ratio of canonical forms of a weakly complete flat front in ${\H}^{3}$}\label{appli-flat}

For a holomorphic Legendrian immersion $\Lc\colon \Sigma \to SL(2, \C)$ on a simply connected Riemann surface $\Sigma$, the projection 
$$
f:= \Lc{\Lc}^{\ast}\colon \Sigma \to {\H}^{3}
$$
gives a {\it flat front} in ${\H}^{3}$. Here, flat fronts in ${\H}^{3}$ are flat surfaces in ${\H}^{3}$ with some admissible singularities  
(see \cite{KRUY2007}, \cite{KUY2004} for the definition of flat fronts in ${\H}^{3}$). We call $\Lc$ the {\it holomorphic lift} of $f$. 
Since $\Lc$ is a holomorphic Legendrian map, ${\Lc}^{-1}{d\Lc}$ is off-diagonal (see \cite{GMM2000}, \cite{KUY2003}, \cite{KUY2004}). 
If we set 
$$
{\Lc}^{-1}{d\Lc} = \left(
\begin{array}{cc}
0      & \theta  \\
\omega & 0
\end{array}
\right), 
$$
then the pull-back of the canonical Hermitian metric of $SL(2, \C)$ by $\Lc$ is represented as 
$$
ds^{2}_{\Lc}:=|\omega|^{2}+|\theta|^{2}
$$ 
for holomorphic $1$-forms $\omega$ and $\theta$ on $\Sigma$. A flat front $f$ is said to be {\it weakly complete} 
if the metric $ds^{2}_{\Lc}$ is complete (\cite{KRUY2009, UY2011}). We define a meromorphic function on $\Sigma$ by the ratio of canonical forms 
$$
\rho := \dfrac{\theta}{\omega}. 
$$
Then a point $p\in \Sigma$ is a singular point of $f$ if and only if $|\rho (p)|=1$ (\cite{KRSUY2005}). We note that 
$$
ds^{2}_{\Lc}=|\omega|^{2}+|\theta|^{2}=(1+|\rho|^{2})|\omega|^{2}. 
$$

Applying Theorem \ref{thm2-1} to the metric $ds^{2}_{\Lc}$, we can get the following theorem.

\begin{theorem}[{\cite[Theorem 4.8]{Ka2013}}]\label{thm3-6}
Let $f\colon \Sigma \to {\H}^{3}$ be a flat front on a simply connected  Riemann surface $\Sigma$. 
Suppose that the ratio of canonical forms $\rho\colon \Sigma \to \RC$ omits greater than or equal to $4$ $(=2+2)$ distinct values. 
Then there exists a positive constant $C$ depending on the set of exceptional values, but not $\Sigma$, such that for all $p\in \Sigma$ we have 
$$
|K_{ds_{\Lc}^{2}}(p)|^{1/2}\leq \dfrac{C}{d(p)}, 
$$
where $K_{ds^{2}_{\Lc}}(p)$ is the Gaussian curvature of the metric $ds^{2}_{\Lc}$ at $p$ and $d(p)$ is the geodesic distance from $p$ to  
the boundary of $\Sigma$. In particular, if the ratio of canonical forms of a weakly complete flat front in ${\H}^{3}$ is nonconstant, 
then it can omit at most $3$ $(=1+2)$ values. 
\end{theorem}

If $\Sigma$ is not simply connected, then we consider that $\rho$ is a meromorphic function on its universal 
covering surface $\widetilde{\Sigma}$. As an application of Theorem \ref{thm3-6}, 
we can obtain a simple proof of the classification of 
complete nonsingular flat surfaces in ${\H}^{3}$. For the proof, see \cite[Corollary 3.5]{Ka2014}. 

\begin{corollary}[\cite{Sa1973, VV1971}]\label{thm3-7}
Any complete nonsingular flat surface in ${\H}^{3}$ must be a horosphere or a hyperbolic cylinder. 
\end{corollary}

Finally, by applying Theorem \ref{thm2-5}, we provide the following unicity theorem for the ratios of canonical forms 
of weakly complete flat fronts in ${\H}^{3}$. 

\begin{theorem}[{\cite[Theorem 4.29]{Ka2015}}]\label{thm3-8}
Let $f\colon \Sigma \to {\H}^{3}$ and $\widehat{f}\colon \widehat{\Sigma}\to {\R}^{3}$ be two flat fronts on simply 
connected Riemann surfaces and assume that there exists a conformal diffeomorphism 
$\Psi\colon \Sigma \to \widehat{\Sigma}$. Let $\rho\colon \Sigma \to \C\cup\{\infty \}$ and 
$\hat{\rho}\colon \widehat{\Sigma}\to \C\cup\{\infty \}$ be the ratio of canonical forms $f(\Sigma)$ and 
$\widehat{f}(\widehat{\Sigma})$ respectively. 
If $\rho \not\equiv \hat{\rho}\circ \Psi$ and either $f(\Sigma)$ or $f(\widetilde{\Sigma})$ is weakly complete, 
then $\rho$ and $\hat{\rho}\circ \Psi$ share at most $5=(1+4)$ distinct values. 
\end{theorem} 

\begin{remark}\label{rmk3-9}
The hyperbolic Gauss map of a weakly complete or complete flat front in ${\H}^{3}$ has also interesting geometric property. 
For more details, see \cite{KRSUY2005}, \cite{KUY2003}, \cite{KUY2004} and \cite{MUY2014}. 
\end{remark}

\section{Further topics}\label{sec-further}
In this section, we give geometric interpretations of the maximal number of exceptional values and unicity theorem 
for the Gauss maps of complete minimal surfaces in ${\R}^{4}$. We also provide an effective estimate for the maximal number 
of exceptional values of the Gauss map of a nonflat complete minimal surface of finite total curvature in ${\R}^{3}$. 

\subsection{Gauss map of a complete minimal surface in ${\R}^{4}$}\label{further-4}

We first give an optimal estimate for the size of the image of the holomorphic map 
$G=(g_{1}, \ldots, g_{n})\colon \Sigma \to (\RC)^{n}:=
\underbrace{\RC\times \cdots \times \RC}_{n}$ on an open Riemann surface $\Sigma$ with the complete conformal metric 
\begin{equation}\label{equ4-0}
ds^{2}= \prod_{i=1}^{n}(1+|g_{i}|^{2})^{m_{i}}|\omega|^{2}. 
\end{equation}

\begin{theorem}[{\cite[Theorem 2.1]{AAIK2017}}]\label{thm4-1}
Let $\Sigma$ be an open Riemann surface with the conformal metric 
$$
ds^{2}=\displaystyle \prod_{i=1}^{n}(1+|g_{i}|^{2})^{m_{i}}|\omega|^{2}, 
$$
where $G=(g_{1}, \ldots , g_{n})\colon \Sigma \to (\RC)^{n}$ is a holomorphic map, $\omega$ is a holomorphic 
$1$-form on $\Sigma$ and each $m_{i}$ $(i=1, \cdots, n)$ is a positive integer. Assume that $g_{i_{1}}, \ldots, g_{i_{k}}$
$(1\leq i_{1}< \cdots <i_{k} \leq n)$ are nonconstant and the others are constant. If the metric $ds^{2}$ is complete and 
each $g_{i_{l}}$ $(l=1, \cdots , k)$ omits $q_{i_{l}}> 2$ distinct values, then we have 
\begin{equation}\label{equ-4-2}
\displaystyle \sum_{l=1}^{k} \dfrac{m_{i_{l}}}{q_{i_{l}}-2}\geq 1. 
\end{equation}
\end{theorem} 

We note that Theorem \ref{thm4-1} also holds for the case where at least one of $m_{1}, \ldots, m_{n}$ is positive and 
the others are zeros. For instance, we assume that $g:= g_{i_{1}}$ is nonconstant and the others are constant. 
If $m:=m_{i_{1}}$ is a positive integer and the others are zeros, then inequality (\ref{equ4-0}) coincides with 
$$
\dfrac{m}{q-2}\geq 1 \, \Longleftrightarrow  \, q \leq m+2,  
$$
where $q:=q_{i_{1}}$. The result corresponds with Corollary \ref{thm2-3}. 

We next give a unicity theorem for the holomorphic map $G=(g_{1}, \ldots, g_{n})\colon \Sigma \to (\RC)^{n}$ on 
an open Riemann surface $\Sigma$ with the complete conformal metric defined by (\ref{equ4-0}). 

\begin{theorem}[{\cite[Theorem 2.1]{HK2017}}]\label{thm4-2}
Let $\Sigma$ be an open Riemann surface with the conformal metric 
$$
ds^{2}=\displaystyle \prod_{i=1}^{n}(1+|g_{i}|^{2})^{m_{i}}|\omega|^{2} 
$$
and $\widehat{\Sigma}$ another open Riemann surface with the conformal metric 
$$
d\hat{s}^{2}= \displaystyle \prod_{i=1}^{n}(1+|\hat{g}_{i}|^{2})^{m_{i}}|\hat{\omega}|^{2}, 
$$
where $\omega$ and $\hat{\omega}$ are holomorphic $1$-forms, $G$ and $\widehat{G}$ are holomorphic maps 
into $(\RC)^{n}$ on $\Sigma$ and $\widehat{\Sigma}$ respectively, 
and each $m_{i}$ $(i=1, \cdots, n)$ is a positive integer. 
We assume that there exists a conformal diffeomorphism $\Psi \colon \Sigma \to 
\widehat{\Sigma}$, and $g_{i_{1}}, \ldots, g_{i_{k}}$ and $\hat{g}_{i_{1}}, \ldots, \hat{g}_{i_{k}}$ $(1\leq i_{1}< \cdots < i_{k}\leq n)$ are 
nonconstant and the others are constant. For each $i_{l}$ $(l=1, \cdots, k)$, we suppose that 
$g_{i_{l}}$ and $\hat{g}_{i_{l}}\circ \Psi$ share $q_{i_{l}}> 4$ distinct values and $g_{i_{l}} \not\equiv \hat{g}_{i_{l}}\circ \Psi$.  
If either $ds^{2}$ or $d\hat{s}^{2}$ is complete, then we have 
\begin{equation}\label{eq-4-3}
\displaystyle \sum_{l=1}^{k} \dfrac{m_{i_{l}}}{q_{i_{l}}-4}\geq 1. 
\end{equation}
\end{theorem}

We remark that Theorem \ref{thm4-2} also holds for the case where at least one of $m_{1}, \ldots, m_{n}$ is positive 
and the others are zeros. For instance, we assume that $g:= g_{i_{1}}$ and $\hat{g}:= \hat{g}_{i_{1}}$ are nonconstant and the others 
are constant. If $m:= m_{i_{1}}$ is a positive integer and the others are zeros, then inequality (\ref{eq-4-3})  coincides with 
$$
\dfrac{m}{q-4}\geq 1 \, \Longleftrightarrow  \, q \leq m+4,  
$$
where $q:= q_{i_{1}}$. The result corresponds with Theorem \ref{thm2-5}. 

We will apply these results to the Gauss maps of complete minimal surfaces in ${\R}^{4}$. We briefly summarize here basic facts 
on minimal surfaces in ${\R}^{4}$. For more details, we refer the reader to \cite{Ch1965, HO1980, HO1985, Os1964}. 
Let $X=(x^{1}, x^{2}, x^{3}, x^{4})\colon \Sigma \to {\R}^{4}$ be an oriented minimal 
surface in ${\R}^4$. By associating a local complex coordinate $z=u+\sqrt{-1}v$ with each positive isothermal coordinate system 
$(u, v)$, $\Sigma$ is considered as a Riemann surface whose conformal metric is the induced metric $ds^{2}$ from ${\R}^{4}$. 
With respect to the local complex coordinate $z=u+\sqrt{-1}v$ of the surface, it holds that 
$$
\bar{\partial} \partial X =0, 
$$
where $\partial =(\partial /\partial u - \sqrt{-1}\partial /\partial v)/2$, $\bar{\partial} 
=(\partial /\partial u + \sqrt{-1}\partial /\partial v)/2$. Hence each ${\phi}_{i}:= \partial x^{i} dz$ ($i=1, 2, 3, 4$) is a 
holomorphic $1$-form on $\Sigma$. If we set 
$$
\omega = {\phi}_{1} -\sqrt{-1} {\phi}_{2}, \qquad g_{1}=\dfrac{{\phi}_{3}+\sqrt{-1}{\phi}_{4}}{{\phi}_{1} -\sqrt{-1} {\phi}_{2}}, 
\qquad g_{2}=\dfrac{-{\phi}_{3}+\sqrt{-1}{\phi}_{4}}{{\phi}_{1} -\sqrt{-1} {\phi}_{2}}, 
$$
then $\omega$ is a holomorphic $1$-form, and $g_{1}$ and $g_{2}$ are meromorphic functions on $\Sigma$. 
Moreover the holomorphic map $G:=(g_{1}, g_{2})\colon \Sigma \to \RC \times \RC$ coincides with the Gauss map of $X(\Sigma)$. 
We remark that the Gauss map of $X(\Sigma)$ in ${\R}^{4}$ is the map from each point of $\Sigma$ to its oriented tangent plane, 
the set of all oriented (tangent) planes in ${\R}^{4}$ is naturally identified with the quadric 
$$
\mathbf{Q}^{2}(\C) =\{[w^{1}: w^{2}: w^{3}: w^{4}] \in \mathbf{P}^{3}(\C) \, ;\, (w^{1})^{2}+\cdots +(w^{4})^{2} = 0\}
$$
in $\mathbf{P}^{3}(\C)$, and $\mathbf{Q}^{2}(\C)$ is biholomorphic to the product of the Riemann spheres $\RC \times \RC$. 
Furthermore the induced metric from ${\R}^{4}$ is given by 
\begin{equation}\label{equ4-4}
ds^{2}= (1+|g_{1}|^{2})(1+|g_{2}|^{2})|\omega|^{2}. 
\end{equation}

Applying Theorem \ref{thm4-1} to the induced metric, we can obtain the Fujimoto theorem for the Gauss map of a 
complete minimal surface in ${\R}^{4}$. 

\begin{theorem}\cite[Theorem I\hspace{-.1em}I]{Fu1988}\label{thm-4-3}
Let $X\colon \Sigma \to {\R}^{4}$ be a complete nonflat minimal surface and 
$G=(g_{1}, g_{2})\colon \Sigma \to \RC \times \RC $ the Gauss map of $X(\Sigma)$. 
\begin{enumerate}
\item[(i)] Assume that $g_{1}$ and $g_{2}$ are both nonconstant and omit $q_{1}$ and $q_{2}$ distinct values respectively. 
If $q_{1}> 2$ and $q_{2}> 2$, then we have 
\begin{equation}\label{equ4-5}
\dfrac{1}{q_{1}-2}+\dfrac{1}{q_{2}-2}\geq 1. 
\end{equation}
\item[(i\hspace{-.1em}i)] If either $g_{1}$ or $g_{2}$, say $g_{2}$, is constant, then $g_{1}$ can omit at most 
$3$ distinct values.  
\end{enumerate}
\end{theorem}
\begin{proof}
We first show (i). Since $g_{1}$ and $g_{2}$ are both nonconstant and $m_{1}=m_{2}=1$ from (\ref{equ4-4}), 
we can prove inequality (\ref{equ4-5}) by Theorem \ref{thm4-1}. Next we show (i\hspace{-.1em}i). 
If we set that $g_{1}$ omits $q_{1}$ values, then we obtain 
$$
\dfrac{1}{q_{1}-2}\geq 1
$$
from Theorem \ref{thm4-1} because $m_{1}=1$. Thus we have $q_{1}\leq 3$. 
\end{proof}

Hence we reveal that the Fujimoto theorem depends on the orders of the factors $(1+|g_{1}|^{2})$ and $(1+|g_{2}|^{2})$ in 
the induced metric from ${\R}^{4}$ and the Euler characteristic of the Riemann sphere $\RC$. 
In \cite{AAIK2017}, we give some applications of this theorem, for example, to provide optimal results for the maximal number of 
exceptional values of the nontrivial part of the Gauss map of a complete minimal Lagrangian surface in the complex 2-space 
${\C}^{2}$ and the generalized Gauss map of a complete nonorientable minimal surface in ${\R}^{4}$. 

In the same way, by Theorem \ref{thm4-2}, we obtain a unicity theorem for the Gauss maps of complete minimal surfaces in ${\R}^{4}$. 

\begin{theorem}\cite[Theorem 1.2]{HK2017}\label{thm-cor}
Let $X\colon \Sigma \to {\R}^{4}$ and $\widehat{X}\colon \widehat{\Sigma} \to {\R}^{4}$ be two nonflat minimal surfaces, 
and $G=(g_{1}, g_{2})\colon \Sigma \to \RC\times \RC$, $\widehat{G}=(\hat{g}_{1}, \hat{g}_{2})\colon \widehat{\Sigma} 
\to \RC\times \RC$ the Gauss maps of $X(\Sigma)$, $\widehat{X}(\widehat{\Sigma})$ respectively. 
We assume that there exists a conformal diffeomorphism $\Psi \colon \Sigma \to \widehat{\Sigma}$ and 
either $X(\Sigma)$ or $\widehat{X}(\widehat{\Sigma})$ is complete. 
\begin{enumerate}
\item[(i)] Assume that $g_{1}, g_{2}, \hat{g}_{1}, \hat{g}_{2}$ are nonconstant and, for each $i$ $(i=1, 2)$, 
$g_{i}$ and $\hat{g}_{i}\circ \Psi$ share $p_{i} >4$ distinct values. 
If $g_{1}\not\equiv \hat{g}_{1}\circ \Psi$ and $g_{2}\not\equiv \hat{g}_{2}\circ \Psi$, then we have
\begin{equation}\label{eq4-6}
\dfrac{1}{p_{1}-4}+\dfrac{1}{p_{2}-4}\geq 1. 
\end{equation}
In particular, if $p_{1}\geq 7$ and $p_{2}\geq 7$, then either $g_{1}\equiv \hat{g}_{1}\circ \Psi$ or 
$g_{2}\equiv \hat{g}_{2}\circ \Psi$, or both hold. 
\item[(i\hspace{-.1em}i)] Assume that $g_{1}, \hat{g}_{1}$ are nonconstant, 
and $g_{1}$ and $\hat{g}_{1}\circ \Psi$ share $p$ distinct values. 
If $g_{1}\not\equiv \hat{g}_{1}\circ \Psi$ and $g_{2}\equiv \hat{g}_{2}\circ \Psi$ is constant, then we have $p\leq 5$. 
In particular, if $p\geq 6$, then $G\equiv \widehat{G}\circ \Psi$.
\end{enumerate}
\end{theorem}

\subsection{Gauss map of a complete minimal surface of finite total curvature in ${\R}^{3}$}\label{further-algebraic}
We review some of the standard facts on complete minimal surfaces of finite total curvature in ${\R}^{3}$. 
Let $X=(x^{1}, x^{2}, x^{3})\colon \Sigma \to {\R}^{3}$ be an oriented minimal surface in ${\R}^{3}$. 
Set ${\phi}_{i}:= \partial x^{i} dz$ ($i=1, 2, 3$). These satisfy 
\begin{enumerate}
\item[(C)] $\sum {\phi}^{2}_{i}=0$: conformal condition, 
\item[(R)] $\sum |{\phi}_{i}|^{2}> 0$: regularity condition, 
\item[(P)] For every loop $\gamma \in H_{1}(\Sigma, \Z)$, $\Re \int_{\gamma} {\phi}_{i} =0$: period condition. 
\end{enumerate}
For the meromorphic function $g$ and holomorphic $1$-form $\omega$ given by (\ref{eq:W-date3}), 
$$
{\phi}_{1}=\dfrac{1}{2}(1-g^{2})\omega, \quad {\phi}_{2}=\dfrac{\sqrt{-1}}{2}(1+g^{2})\omega, \quad {\phi}_{3}= g\omega
$$
hold. We call $(\omega, g)$ the Weierstrass data (W-data, for short). If we are given the W-data on $\Sigma$, we get ${\phi}_{j}$'s 
by this formula. They satisfy condition (C) automatically, and condition (R) is interpreted as the poles of $g$ of order $l$ 
coincides exactly with the zeros of $\omega$ of order $2k$, because the induced metric $ds^{2}$ is given by (\ref{equ-412-metric}). 
In general, for a given meromorphic function $g$ on $\Sigma$, it is not so hard to find a holomorphic 1-form $\omega$ satisfying 
condition (R). However, the period condition (P) always causes trouble. The total curvature of $X(\Sigma)$ is given by 
$$
\tau (\Sigma) :=\displaystyle \int_{\Sigma} K_{ds^{2}}dA = -\int_{\Sigma} \dfrac{2\sqrt{-1}dg\wedge d\bar{g}}{(1+|g|^{2})^{2}}, 
$$
where $dA$ is the area element with respect to the metric $ds^{2}$. Note that $|\tau (\Sigma)|$ is the area of $\Sigma$ with respect to 
the metric induced from the Fubini-Study metric of the Riemann sphere $\RC$ by $g$.  

\begin{theorem}\label{thm-4-2-1}
A complete minimal surface of finite total curvature $X\colon \Sigma \to {\R}^{3}$ satisfies 
\begin{enumerate}
\item[(i)] $\Sigma$ is conformally to $\overline{\Sigma}_{\gamma}\backslash \{ p_{1},\ldots ,p_{k} \}$, 
where $\overline{\Sigma}_{\gamma}$ is a closed Riemann surface of genus $\gamma$ and 
$p_{1}, \ldots, p_{k}\in \overline{\Sigma}$ $($\cite{Hu1957}$)$, 
\item[(i\hspace{-.1em}i)]  The W-date $(\omega, g)$ can be extended meromorphically to 
$\overline{\Sigma}_{\gamma}$ $($\cite{Os1964}$)$. 
\end{enumerate}
\end{theorem}

From this fact, we call such surfaces {\it algebraic minimal surfaces}. 
Osserman proved the following result for the number of exceptional values of the Gauss map of a complete minimal surface of finite total curvature in ${\R}^{3}$. 

\begin{theorem}\cite[Theorem 3]{Os1964}\label{thm4-2-2}
The Gauss map of a nonflat complete minimal surface of finite total curvature omits at most 3 values. 
\end{theorem}

The author, Kobayashi and Miyaoka refined Theorem \ref{thm4-2-2} and give the following estimate for the number of exceptional values of the Gauss map of a complete minimal surface 
of finite total curvature in ${\R}^{3}$. 

\begin{theorem}\cite[Theorem 3.3]{KKM2008}\label{thm4-2-3}
Let $X\colon \Sigma= \overline{\Sigma}_{\gamma}\backslash \{ p_{1},\ldots ,p_{k} \}\to {\R}^{3}$ be a nonflat complete minimal surface 
of finite total curvature, $g\colon \Sigma\to \RC$ its Gauss map, $d$ the degree of $g$ considered as a map on $\overline{\Sigma}_{\gamma}$. Then the number $D_{g}$ of exceptinal values of $g$ satisfies 
\begin{equation}\label{eq4-2-1}
D_{g}\leq 2+\dfrac{2}{R}, \qquad R=\dfrac{d}{\gamma -1+(k/2)}> 1. 
\end{equation}
\end{theorem}

\begin{proof}
By a suitable rotation of the surface in ${\R}^{3}$, we may assume that the Gauss map $g$ has neither zero nor pole at $p_{j}$ 
and that the zeros and poles of $g$ are simple. The simple poles of $g$ coincide with the double zeros of $\omega$ because 
the surface satifies the regularity condition (R). By the completeness of the surface, 
$\omega$ has a pole at each end $p_{j}$ (\cite{Ma1963}, 
\cite[Lemma 9.6]{Os1986}). Moreover, since the surface satisfies the period condition (P), 
$\omega$ has a pole of order ${\mu}_{j}\geq 2$ at $p_{j}$ (\cite{Os1964}). 
Applying the Riemann-Roch theorem to $\omega$ on $\overline{\Sigma}_{\gamma}$, we obtain that 
$$
\displaystyle 2d-\sum_{j=1}^{k}{\mu}_{j} = 2\gamma -2.  
$$
Thus we have
\begin{equation}\label{eq4-2-2}
\displaystyle d=\gamma -1+\dfrac{1}{2}\sum_{j=1}^{k}{\mu}_{j}\geq \gamma -1+k > \gamma -1+(k/2), 
\end{equation}
and $R>1$. 

On the other hand, we assume that $g$ omits $D_{g}$ values. Let $n_{0}$ be the sum of the branching orders at the image of 
exceptional values. Then we have 
$$
k\geq dD_{g} -n_{0}. 
$$
Let $n_{g}$ be the total branching order of $g$ on $\overline{\Sigma}_{\gamma}$. Then applying the Riemann-Hurwitz formula to 
the meromorphic function $g$ on $\overline{\Sigma}_{\gamma}$, we have
\begin{equation}\label{eq4-2-3}
n_{g}=2(d+\gamma -1). 
\end{equation}
Hence we have
$$
D_{g}\leq \dfrac{n_{0}+k}{d}\leq \dfrac{n_{g}+k}{d}=2+\dfrac{2}{R}. 
$$
\end{proof}

\begin{remark}\label{rmk4-2-4}
More precisely, (\ref{eq4-2-1}) holds for the totally ramified value number ${\nu}_{g}$ for the Gauss map of 
a complete minimal surface of finite total curvature in ${\R}^{3}$. One of the most important results for the number is to discover 
nonflat complete minimal surfaces of finite total curvature in ${\R}^{3}$ with ${\nu}_{g}=2.5$. For the details, see \cite{Ka2006}. 
\end{remark}

By the proof of Theorem \ref{thm4-2-3}, we reveal that the reason why the upper bound for $D_{g}$ changes from `4' to `3' is that 
the order $\mu_{j}$ of a pole of $\omega$ at each end $p_{j}$ changes from `$\mu_{j}\ge 1$' to `$\mu_{j}\ge 2$'. 
Remark that this principle is equivalent to the distinction between the Cohn-Vossen inequality and the Osserman inequality on the total 
curvature of a complete minimal surface of finite total curvature in ${\R}^{3}$. 

There still remains the following question. 

\begin{problem}[\cite{Os1964}]\label{pro4-2-5}
Does there exist a complete minimal surface of finite total curvature in ${\R}^{3}$ whose Gauss map 
omits 3 values?
\end{problem} 
If so, Theorems \ref{thm4-2-2} and \ref{thm4-2-3} are optimal. If not, the maximum is `2' and is attained by the catenoid 
and examples constructed by Miyaoka and Sato \cite{MS1994}. In regards to this problem, the following facts are well-known. 

\begin{proposition}\label{pro4-2-6}
For a nonflat complete minimal surface of finite total curvature in ${\R}^{3}$, 
\begin{enumerate}
\item[(i)] When $\gamma =0$, the Gauss map omits at most 2 values $($\cite{Os1964}, \cite{KKM2008}$)$, 
\item[(i\hspace{-.1em}i)] When $\gamma =1$ and the surface has a non-embedded end, the Gauss map omits at most 2 values 
$($\cite{Ga1976}, \cite{Fa1993}, \cite{KKM2008}$)$, 
\item[(i\hspace{-.1em}i\hspace{-.1em}i)] If the Gauss map omits 3 values, then $\gamma \geq 1$ and 
the total curvature $\tau (\Sigma)\leq -20\pi$ $($\cite{Os1964}, \cite{WX1987}, \cite{Fa1993}$)$. 
\end{enumerate} 
\end{proposition}

By virtue of Theorem \ref{thm4-2-3} and this proposition, if there exists a complete minimal surface of 
finite total curvature whose Gauss map omits 3 values, then it has the complexity of topological data. 
Thus it is very hard to solve the period condition (P) of the surface. 
This is the difficulty of this problem. 



\begin{thebibliography}{99}
\bibitem{AAIK2017} 
R. Aiyama, K. Akutagawa, S. Imagawa and Y. Kawakami, Remarks on the Gauss images of complete minimal surfaces in 
Euclidean four-space, to appear in Annali di Matematica Pura ed Applicata, (DOI: 10.1007/s10231-017-0643-6).  

\bibitem{Ah1935}
L. V. Ahlfors, 
Zur Theorie der \"Uberlagerungsfl\"achen, Acta Math., {\bf 65} (1935) 157--194, and Collected Papers Vol. I, pp. 163--173. 

\bibitem{Be1915}
S. Bernstein, Sur un th\'eor\`eme de g\'eom\'etrie et ses applications aux \'equations aux d\'eriv\'ees 
partielles du type elliptique, Comm.~de la Soc.~Math. de Kharkov (2\'eme s\'er.) {\bf 15}, 38--45 (1915--1917). 

\bibitem{Be2000}
W. Bergweiler, The role of the Ahlfors five islands theorem in complex dynamics, 
Conform. Geom. Dyn., {\bf 4} (2000), 22--34. 

\bibitem{Ca1958}
E. Calabi, 
Improper affine hypersurfaces of convex type and a generalization of a theorem by K.~J\"orgens, 
Mich. Math. J., {\bf 5} (1958), 108--126. 

\bibitem{Ca1970}
E. Calabi, 
Examples of Bernstein problems for some nonlinear equations, 
In: Global Analysis, Berkeley, CA, 1968 (eds. S.-S. Chern and S. Smale), 
Proc. Sympos. Pure Math., {\bf 15}, Amer. Math. Soc. Providence, RI, pp. 223--230. 

\bibitem{Ch1960}
S.-S. Chern, 
Complex analytic mappings of Riemann surfaces. I, 
Amer. J. Math., {\bf 82} (1960), 323--337. 

\bibitem{Ch1965} 
S. S. Chern, 
Minimal surfaces in an Euclidean space of $N$ dimensions, 
1965 Differential and Combinatorial Topology (A Symposium in Honor of Marston Morse), 187--198, 
Princeton Univ. Press, Princeton, N.J. 

\bibitem{CM1987}
B.-Y. Chen and J.-M. Morvan, 
G\'eom\'etrie des surfaces lagrangiennes de ${\C}^{2}$, 
J. Math. Pures Appl., {\bf 66} (1987), 321--325. 

\bibitem{Fa1993}
Y. Fang, 
On the Gauss map of complete minimal surfaces with finite total curvature, 
Indiana Univ. Math. J., {\bf 42} (1993), 1389--1411.  

\bibitem{Fu1988}
H. Fujimoto, 
On the number of exceptional values of the Gauss map of minimal surfaces, 
J. Math. Soc. Japan, {\bf 40} (1988), 237--249. 

\bibitem{Fu1993}
H. Fujimoto, 
Unicity theorems for the Gauss maps of complete minimal surfaces, 
J. Math. Soc. Japan, {\bf 45} (1993), 481--487. 

\bibitem{Fu1993-2}
H. Fujimoto, 
Value distribution theory of the Gauss map of minimal surfaces in ${\R}^{m}$, 
Aspects of Mathematics, E21. Friedr. Vieweg \& Sohn, Braunschweig, 1993. 

\bibitem{Fu1997}
H. Fujimoto,
Nevanlinna theory and minimal surfaces, In: Geometry V: Minimal Surfaces (ed. R. Osserman), 
Encyclopaedia Math. Sci., {\bf 90}, Springer, Berlin, 1997, pp. 95--151; pp. 267--272. 

\bibitem{Ga1976}
F. Gackst\"atter, 
\"Uber abelsche Minimalfl\"achen, Math. Nachr., {\bf 74} (1976), 165--165. 

\bibitem{GMM2000}
J. A. G\'alvez, A. Mart\'inez and F. Mil\'an, 
Flat surfaces in hyperbolic 3-space, 
Math. Ann., {\bf 316} (2000), 419--435. 

\bibitem{HK2017}
P. H. Ha and Y. Kawakami, 
A note on a unicity theorem for the Gauss maps of complete minimal surfaces in Euclidean four-space, 
to appear in Canadian Mathematical Bulletin, (DOI: 10.4153/CMB-2017-015-0). 

\bibitem{HO1980}
D. A. Hoffman and R. Osserman, 
The geometry of the generalized Gauss map, 
Mem. Amer. Math. Soc. {\bf 28} (1980), no. 236. 

\bibitem{HO1985}
D. A. Hoffman and R. Osserman, 
The Gauss map of surfaces in ${\R}^{3}$ and ${\R}^4$, 
Proc. London Math. Soc., {\bf 50} (1985), 27--56. 

\bibitem{Hu1957}
A. Huber, 
On subharmonic functions and differential geometry in the large, 
Comment. Math. Helv., {\bf 32} (1957), 13--72. 

\bibitem{Jo1954}
K. J\"orgens, 
\"Uber die L\"osungen der differentialgleichung $rt-s^{2}=1$, 
Math. Ann., {\bf 127} (1954), 130--134. 

\bibitem{Ka2006}
Y. Kawakami, 
On the totally ramified value number of the Gauss map of minimal surfaces, 
Proc. Japan Acad. Ser. A Math. Sci., {\bf 82} (2006), 1--3. 

\bibitem{Ka2013}
Y. Kawakami, 
On the maximal number of exceptional values of Gauss maps for various classes of surfaces, 
Math. Z., {\bf 274} (2013), 1249--1260. 

\bibitem{Ka2014}
Y. Kawakami, 
A ramification theorem for the ratio of canonical forms of flat surfaces in hyperbolic three-space, 
Geom. Dedicata, {\bf 171} (2014), 387--396. 

\bibitem{Ka2015}
Y. Kawakami, 
Function-theoretic properties for the Gauss maps of various classes of surfaces, 
Canad. J. Math., {\bf 67} (2015), 1411--1434. 

\bibitem{KKM2008}
Y. Kawakami, R. Kobayashi and R. Miyaoka, 
The Gauss map of pseudo-algebraic minimal surfaces, 
Forum Math., {\bf 20} (2008), 1055--1069. 

\bibitem{KN2012}
Y. Kawakami and D. Nakajo, 
Value distribution of the Gauss map of improper affine spheres, 
J. Math. Soc. Japan, {\bf 64} (2012), 799--821. 

\bibitem{Ko2003}
R. Kobayashi, 
Toward Nevanlinna theory as a geometric model for Diophantine approximation,  
Sugaku Expositions, {\bf 16} (2003), 39--79. 

\bibitem{KRSUY2005}
M. Kokubu, W. Rossman, K. Saji, M. Umehara and K. Yamada, 
Singularities of flat fronts in hyperbolic space, 
Pacific J. Math., {\bf 221} (2005), 303--351. 

\bibitem{KRUY2007}
M. Kokubu, W. Rossman, M. Umehara and K. Yamada, 
Flat fronts in hyperbolic 3-space and their caustics, 
J. Math. Soc. Japan, {\bf 59} (2007), 265--299. 

\bibitem{KRUY2009}
M. Kokubu, W. Rossman, M. Umehara and K. Yamada, 
Asymptotic behavior of flat surfaces in hyperbolic 3-space, 
J. Math. Soc. Japan, {\bf 61} (2009), 799--852. 

\bibitem{KUY2003}
M. Kokubu, M. Umehara and K. Yamada, 
An elementary proof of Small's formula for null curves in PSL(2, $\C$) and an analogue for Legendrian curves in PSL(2, $\C$), 
Osaka J. Math. {\bf 40} (2003), 697--715. 

\bibitem{KUY2004}
M. Kokubu, M. Umehara and K. Yamada, 
Flat fronts in hyperbolic 3-space, Pacific J. Math., {\bf 216} (2004), 149--175. 

\bibitem{La1982}
H. B. Lawson Jr., 
Lectures on minimal submanifolds, Vol. I, Mathematics Lecture Series 9, Publish or Perish, Inc., 1980. 

\bibitem{Ma1963}
G. R. Maclane, 
On asymptotic values, abstract 603--166, Notices Amer. Math. Soc., {\bf 10} (1963), 482--483. 

\bibitem{Ma2005}
A. Mart\'inez, 
Improper affine maps, Math. Z., {\bf 249} (2005), 755--766. 

\bibitem{MP2012}
W. H. Meeks III and J. P\'erez, 
A survey on classical minimal surface theory. 
University Lecture Series, {\bf 60}, American Mathematical Society, Providence, RI, 2012. 

\bibitem{MS1994}
R. Miyaoka and K. Sato, 
On complete minimal surfaces whose Gauss map misses two directions, 
Arch. Math., {\bf 63} (1994), 565--576. 

\bibitem{MUY2014}
F. Mat\'in, M. Umehara and K. Yamada, 
Flat surfaces in hyperbolic 3-space whose hyperbolic Gauss maps are bounded, 
Rev. Mat. Iberoam., {\bf 30} (2014), 309--316. 

\bibitem{Na2009}
D. Nakajo, 
A representation formula for indefinite improper affine spheres, 
Results Math., {\bf 55} (2009), 139--159. 

\bibitem{Ne1926} 
R. Nevanlinna, 
Einige Eindeutigkeitss\"atze in der Theorie der Meromorphen Funktionen, 
Acta Math., {\bf 48} (1926), 367--391. 

\bibitem{Ne1970} 
R. Nevanlinna, 
Analytic functions, Translated from the second German edition by Phillip Emig. Die Grundlehren der mathematischen Wissenschaften, 
Band {\bf 162}, Springer-Verlag, New York-Berlin, 1970. 

\bibitem{NO1990}
J. Noguchi and T. Ochiai, 
Geometric function theory in several complex variables, Transl. Math. Monogt. {\bf 80}, 
Amer. Math. Soc., Providence, RI, 1990. 

\bibitem{NW2013}
J. Noguchi and J. Winkelmann, 
Nevanlinna theory in several complex variables and Diophantine approximation, 
Grundlehren der mathematischen Wissenschaften, Vol. 350, Springer, 2014. 

\bibitem{Os1964}
R. Osserman, 
Global properties of minimal surfaces in $E^{3}$ and $E^{n}$, 
Ann. of Math., {\bf 80} (1964), 340--364. 

\bibitem{Os1986}
R. Osserman, 
A survey of minimal surfaces, 2nd edition, Dover Publication Inc., New York, 1986. 

\bibitem{Ro2002}
A. Ros, 
The Gauss map of minimal surfaces, 
Differential geometry, Valencia, 2001,  235--252, World Sci. Publ., 2002.

\bibitem{Ru2001}
M. Ru, 
Nevanlinna theory and its relation to Diophantine approximation, World Sci., River Edge, NJ, 2001. 

\bibitem{Sa1973}
S. Sasaki, 
On complete flat surfaces in hyperbolic 3-space, 
K\={o}dai Math Sem. Rep., {\bf 25} (1973), 449--457.

\bibitem{SUY2009}
K. Saji, M. Umehara and K. Yamada, 
The geometry of fronts, 
Ann. of Math., {\bf 169} (2009), 491--529. 

\bibitem{UY2011}
M. Umehara and K. Yamada, 
Applications of a completeness lemma in minimal surface theory to various classes of surfaces, 
Bull. London Math. Soc., {\bf 43} (2011), 191--199. 
Corrigendum, Bull. London Math. Soc., {\bf 44} (2012), 617--618.

\bibitem{VV1971}
Y. A. Volkov and S. M. Vladimirova, 
Isometric immersions of the Euclidean plane in Loba\u{c}evski\u{i} space (Russian), 
Mat. Zametki, {\bf 10} (1971), 327--332. 

\bibitem{WX1987}
A. Weitsman and F. Xavier, 
Some function theoretic properties of the Gauss map for hyperbolic complete minimal surfaces, 
Mich. Math. J., {\bf 34} (1987), 275--283. 

\end{thebibliography}
\end{document}